\documentclass[11pt, english]{article}
\usepackage[margin=3cm]{geometry}
\usepackage{amsbsy,amsmath,amsthm,amssymb}

\usepackage{fourier}

\usepackage{xcolor,bm}
\usepackage{esint}
\usepackage{babel}
\usepackage[colorlinks=true, citecolor = blue]{hyperref}

\usepackage{indentfirst}
\usepackage{changepage}
\usepackage{setspace}
\usepackage{indentfirst}
\usepackage{graphicx}
\usepackage{calc}

\allowdisplaybreaks

\theoremstyle{plain}
\newtheorem{theorem}{Theorem}[section]

\theoremstyle{remark}
\newtheorem{remark}[theorem]{Remark}

\theoremstyle{definition}
\newtheorem{definition}{Definition}

\newenvironment{proof of theorem 1.1}{{\noindent \em Proof of Theorem 1.1.}}{\hfill $\Box$\par}
\newenvironment{proof of theorem 1.2}{{\noindent \em Proof of Theorem 1.2.}}{\hfill $\Box$\par}

\DeclareSymbolFont{EulerExtension}{U}{euex}{m}{n}
\DeclareMathSymbol{\euintop}{\mathop} {EulerExtension}{"52}
\DeclareMathSymbol{\euointop}{\mathop} {EulerExtension}{"48}

\begin{document}
	\title{Boundedness of certain multiple Erd\'{e}lyi-Kober fractional integral operators on the Hardy space $H^1$}
	\author{Xi Chen$^{\rm 1}$, Min-Jie Luo$^{\rm 2}$\thanks{Corresponding author}}
	\date{}
	\maketitle
	\begin{center}\small
		$^{1}$\emph{Department of Mathematics, School of Mathematics and Statistics, \\ Donghua University, Shanghai 201620, \\ 
			People's Republic of China.}\\
		E-mail: \texttt{ximath@163.com}
		\vspace{0.3cm}
		
		$^{2}$\emph{Department of Mathematics, School of Mathematics and Statistics, \\ Donghua University, Shanghai 201620, \\ 
			People's Republic of China.}\\
		E-mail: \texttt{mathwinnie@live.com}, \texttt{mathwinnie@dhu.edu.cn}
	\end{center}
	
	\vspace{0.5cm}
	
	\begin{abstract}
		
		In this paper, we establish the boundedness of the multiple Erd\'{e}lyi-Kober fractional integral operators involving Fox's $H$-function on the Hardy space $H^1$. Our results generalize recent results of Kwok-Pun Ho [Proyecciones 39 (3) (2020), 663--677]. Some useful connections related to the Hausdorff operators are also mentioned. \\
		
		\\
		
	\end{abstract}

\section{Introduction}

The boundedness of fractional integral operators on various function spaces plays a fundamental role in the theory of Fractional Calculus. For common fractional integral operators such as the famous Riemann-Liouville fractional integral operators, some useful results have already been obtained (see, for example, \cite[Chapter 2]{Kilbas-Srivastava-Trujillo-2006}). However, for fractional integral operators arising from the Generalized Fractional Calculus --- typically those with integral kernels involving special functions --- the known results remain very limited. 

In Generalized Fractional Calculus, a typical class of operators is the multiple Erd\'{e}lyi-Kober operators introduced by Kiryakova, which are closely related to Meijer's $G$-function. For any $m\in\mathbb{Z}$, let $\mathbb{Z}_{\geq m}:=\{n\in\mathbb{Z};n\geq m\}$. Let $m, n, p, q \in \mathbb{Z}_{\geq0}$, $0 \leq m \leq q$, $0 \leq n \leq p$ and let $a_1,\cdots,a_p,b_1,\cdots, b_q\in \mathbb{C}$. Then, Meijer's $G$-function is defined by
\begin{align}\label{G-Function-Def}
	G_{p, q}^{m, n}\left[z \left|\begin{matrix}
		(a_i)_{1,p} \\
		(b_j)_{1,q}
	\end{matrix}\right.\right]
	&=G_{p, q}^{m, n}\left[z \left|\begin{matrix}
		a_1, \cdots, a_p \\
		b_1, \cdots, b_q
	\end{matrix}\right.\right] \notag\\
	&=\frac{1}{2\pi\mathrm{i}} \int_L \frac{\displaystyle \prod_{k=1}^m \Gamma(s+b_k) \prod_{i=1}^n \Gamma(1-a_i-s)}{\displaystyle \prod_{k=m+1}^q \Gamma(1-b_k-s) \prod_{i=n+1}^p \Gamma(s+a_i)}  z^{-s} \mathrm{d}s,
\end{align}
where $L$ is a suitable contour that separates the poles of  $\Gamma(s+b_k)$ from the poles of  $\Gamma(1-a_i-s)$ (see, for example, \cite[p. 313]{Kiryakova-book-1994} and \cite[p. 1]{Mathai-Saxena-Book-1973}). The multiple Erd\'{e}lyi-Kober operators $I_{\beta,m}^{(\gamma_k),(\delta_k)}$ and $W_{\beta,m}^{(\gamma_k),(\delta_k)}$ are defined by (see \cite[p. 11]{Kiryakova-book-1994} and \cite[p. 12]{Srivastava-Saxena-2001})
\begin{equation}\label{multiE-K-Def-1}
I_{\beta,m}^{(\gamma_k),(\delta_k)}f(x)
=\int_{0}^{1}G_{m,m}^{m, 0}\left[\sigma \left|\begin{matrix}
	(\gamma_k+\delta_k)_{1,m} \\
	(\gamma_k)_{1,m}
\end{matrix}\right.\right]f(x\sigma^{1/\beta})\mathrm{d}\sigma
\end{equation}
and
\begin{equation}\label{multiE-K-Def-2}
	W_{\beta,m}^{(\gamma_k),(\delta_k)}f(x)
	=\int_{1}^{\infty}G_{m,m}^{m, 0}\left[\frac{1}{\sigma} \left|\begin{matrix}
		(\gamma_k+\delta_k+1)_{1,m} \\
		(\gamma_k+1)_{1,m}
	\end{matrix}\right.\right]f(x\sigma^{1/\beta})\mathrm{d}\sigma,
\end{equation}
respectively. The $L^p$-boundedness of $I_{\beta,m}^{(\gamma_k),(\delta_k)}$ and $W_{\beta,m}^{(\gamma_k),(\delta_k)}$ is presented in \cite[p. 37, Theorem 1.2.17]{Kiryakova-book-1994}. 
Recently, Ho \cite{Ho-2020} studied the boundedness of the multiple Erd\'{e}lyi-Kober fractional integral operators \eqref{multiE-K-Def-1} and \eqref{multiE-K-Def-2} on the Hardy space $H^{1}$. 

Let $\mathcal{S}$ and $\mathcal{S}^{\prime}$ denote the class of Schwartz functions and the class of tempered distributions on $\mathbb{R}$, respectively. The convolution of two functions $f*g$ is defined as usual by
\[
(f*g)(x):=\int_{\mathbb{R}}f(x-y)g(y)\mathrm{d}y
=\int_{\mathbb{R}}f(y)g(x-y)\mathrm{d}y
\]
whenever the expression makes sense. We say that $f\in\mathcal{S}'$ is bounded if $f*\varphi\in L^{\infty}$ for any $\varphi\in\mathcal{S}$. 
The Poisson kernel $P$ is the function
\[
P(x)=\frac{1}{\pi}\frac{1}{1+|x|^2}, ~x\in\mathbb{R}.
\]
For $t>0$, let $P_t(x):=t^{-1}P(x/t)$. In addition, we write $D_{\lambda}f(x)=f(x/\lambda)$ $(x \in \mathbb{R})$ for any $\lambda >0$.
\begin{definition}[{\cite[p. 56]{Grafakos-Book-2014}}]
	Let $f\in\mathcal{S}'$ be bounded. We say that $f$ lies in the Hardy space $H^1(\mathbb{R})$ if the Poisson maximal function \begin{equation}\label{Def-MP}
		M_P f(x)=\sup _{t>0}\left|(P_t *f)(x)\right|
	\end{equation}
	lies in $L^1(\mathbb{R})$. If this is the case, we set
	\begin{equation}
		\|f\|_{H^1}=\|M_P f\|_{L_1}.
	\end{equation}
\end{definition}
The Hardy space $H^1$ has the following useful property (see \cite[p. 57, Theorem 2.1.2 (b)]{Grafakos-Book-2014} and \cite[p. 91]{Stein-Book-1993})
\begin{equation}
	H^1 \subset L^1.
\end{equation}
Now we can state Ho's main result. 
\begin{theorem}[{\cite[p. 669]{Ho-2020}}]\label{HoTh}
	Let $m\in\mathbb{Z}_{\geq1}$, $\beta>0$, $\gamma_i\in\mathbb{R}$ and $\delta_i>0$.
	\begin{itemize}
		\item[\emph{(1)}] If
		$\displaystyle \min_{1\leq k\leq m}\gamma_k>\frac{1}{\beta}-1$ and $\displaystyle \sum_{k=1}^{m}\delta_k>0$, then for any $f\in H^1$, 
		\[
		\|I_{\beta,m}^{(\gamma_k),(\delta_k)}f\|_{H^1}
		\leq \mathsf{c}_1\|f\|_{H^1},
		\]
		where
		\[
		\mathsf{c}_1:=\int_{0}^{1}\left|G_{m,m}^{m, 0}\left[\sigma \left|\begin{matrix}
			(\gamma_k+\delta_k)_{1,m} \\
			(\gamma_k)_{1,m}
		\end{matrix}\right.\right]\right|\sigma^{-1/\beta}\mathrm{d}\sigma. 
		\]
		\item[\emph{(2)}] If
		$\displaystyle \min_{1\leq k\leq m}\gamma_k>-\frac{1}{\beta}$, then for any $f\in H^1$, 
		\[
		\|W_{\beta,m}^{(\gamma_k),(\delta_k)}f\|_{H^1}
		\leq \mathsf{c}_2\|f\|_{H^1},
		\]
		where
		\[
		\mathsf{c}_2:=\int_{1}^{\infty}\left|G_{m,m}^{m, 0}\left[\frac{1}{\sigma}\left|\begin{matrix}
			(\gamma_k+\delta_k+1)_{1,m} \\
			(\gamma_k+1)_{1,m}
		\end{matrix}\right.\right]\right|\sigma^{-1/\beta}\mathrm{d}\sigma. 
		\]
	\end{itemize}
	
\end{theorem}

In this paper, we study the $H^1$-boundedness of the following multiple Erd\'{e}lyi-Kober fractional integral operators introduced by Galu\'{e} \emph{et al.} \cite{Galue-Kalla-Srivastava-1993}: 
\begin{equation}\label{GalueFIO-Def-1}
	I_{(\beta_k),(\lambda_k),m}^{(\gamma_k),(\delta_k)}f(x)
	:=\begin{cases}
		\displaystyle \int_0^1 H_{m,m}^{m,0}\left[u\left|\begin{matrix}
			(\gamma_k+\delta_k+1-1/\beta_k,1/\beta_k)_{1,m}\\
			(\gamma_k+1-1/\lambda_k,1/\lambda_k)_{1,m}
		\end{matrix}\right.\right]f(xu)\mathrm{d}u, ~ \text{if}~\sum_{k=1}^{m}\delta_k>0,\\
		f(x), ~\text{if}~ \delta_k=0~\text{and}~\lambda_k=\beta_k~(k=1,\cdots,m),
	\end{cases}
\end{equation}
where $m\in\mathbb{Z}_{\geq1}$, $\beta_k>0$, $\lambda_k>0$, $\delta_k\geq0$, and $\gamma_k\in\mathbb{R}$ $(k=1,\cdots,m)$.
\begin{equation}\label{GalueFIO-Def-2}
	K_{(\varepsilon_k),(\xi_k),n}^{(\tau_k),(\alpha_k)}f(x)
	:=\begin{cases}
		\displaystyle \int_1^\infty H_{n,n}^{n,0}\left[\frac{1}{u}\left|\begin{matrix}
			(\tau_k+\alpha_k+1/\varepsilon_k,1/\varepsilon_k)_{1,n}\\
			(\tau_k+1/\xi_k,1/\xi_k)_{1,n}
		\end{matrix}\right.\right]f(xu)\mathrm{d}u, ~ \text{if}~\sum_{k=1}^{n}\alpha_k>0,\\
		f(x), ~\text{if}~ \alpha_k=0~\text{and}~\varepsilon_k=\xi_k~(k=1,\cdots,n),
	\end{cases}
\end{equation}
where $n\in\mathbb{Z}_{\geq1}$, $\varepsilon_k>0$, $\xi_k>0$, $\alpha_k\geq0$, and $\tau_k\in\mathbb{R}$ $(k=1,\cdots,m)$. In the present study, we focus on the case when $\sum_{k=1}^{m}\delta_k>0$ in \eqref{GalueFIO-Def-1} and the case when $\sum_{k=1}^{n}\alpha_k>0$ in \eqref{GalueFIO-Def-2}. Here, $H_{m,m}^{m,0}$ and $H_{n,n}^{n,0}$ represent particular cases of the Fox $H$-function, which is defined analogously to the Meijer $G$-function through Mellin-Barnes integrals. Let $m, n, p, q \in \mathbb{Z}_{\geq0}$, $0 \leq m \leq q$, $0 \leq n \leq p$ and let $a_{i}, b_j\in \mathbb{C}, A_i, B_j\geq0$, where $1\leq i \leq p, 1\leq j \leq q$. The Fox $H$-function is defined by (see \cite[p. 1]{Kilbas-Saigo-Book-2004} and \cite[p. 58]{Kilbas-Srivastava-Trujillo-2006}):
\begin{align}\label{H-Function-Def}
	H_{p, q}^{m, n}\left[z \left|\begin{matrix}
		(a_i, A_i)_{1,p} \\
		(b_j, B_j)_{1,q}
	\end{matrix}\right.\right]
	&=H_{p, q}^{m, n}\left[z \left|\begin{matrix}
		(a_1, A_1), \ldots,(a_p, A_p) \\
		(b_1, B_1), \cdots,(b_q, B_q)
	\end{matrix}\right.\right] \notag\\
	&=\frac{1}{2 \pi \mathrm{i}} \int_L \frac{\displaystyle\prod_{k=1}^m \Gamma(B_k s+b_k) \prod_{i=1}^n \Gamma(1-a_i-A_i s)}{\displaystyle \prod_{k=m+1}^q \Gamma(1-b_k-B_k s) \prod_{i=n+1}^p \Gamma(A_i s+a_i)}  z^{-s} \mathrm{d}s,
\end{align}
where $L$ is a suitable contour that separates the poles of  $\Gamma\left(B_k s+b_k\right)$ from the poles of  $\Gamma\left(1-a_i-A_i s\right)$. To further clarify the definition, we take $L=L_{-\infty}$ and require that
\begin{equation}\label{H-function-Condition-1}
	\Delta:=\sum_{i=1}^p A_i-\sum_{j=1}^q B_j=0.
\end{equation}
Here, $L_{-\infty}$ denotes a left loop situated in a horizontal strip starting at the point $-\infty+\mathrm{i}\varphi_1$ and terminating at the point $-\infty+\mathrm{i}\varphi_2$ with $-\infty<\varphi_1<\varphi_2<\infty$. Then the $H$-function defined by \eqref{H-Function-Def} makes sense for $0<|z|<\delta$, where 
\begin{equation}\label{H-function-Condition-2}
\delta:=\prod_{j=1}^{p}A_j^{-A_j}\prod_{j=1}^{q}B_j^{B_j}.
\end{equation}

Recall that (see \cite[p. 31, Property 2.4]{Kilbas-Saigo-Book-2004})
\[
H_{p, q}^{m, n}\left[z \left|\begin{matrix}
	(a_i, 1/\beta)_{1,p} \\
	(b_j, 1/\beta)_{1,q}
\end{matrix}\right.\right]
=\beta G_{p, q}^{m, n}\left[z^{\beta} \left|\begin{matrix}
	(a_i)_{1,p} \\
	(b_j)_{1,q}
\end{matrix}\right.\right]~(\beta>0)
\]
and \cite[p. 32, Property 2.5]{Kilbas-Saigo-Book-2004}
\[
z^{\rho}G_{p, q}^{m, n}\left[z\left|\begin{matrix}
	(a_i)_{1,p} \\
	(b_j)_{1,q}
\end{matrix}\right.\right]
=G_{p, q}^{m, n}\left[z\left|\begin{matrix}
	(a_i+\rho)_{1,p} \\
	(b_j+\rho)_{1,q}
\end{matrix}\right.\right]~(\rho\in\mathbb{C}).
\]
By letting $\beta_k=\lambda_k=\beta$ ($k=1,\cdots,m$) in \eqref{GalueFIO-Def-1}, we obtain 
\begin{align*}
I_{(\beta,\cdots,\beta),(\beta,\cdots,\beta),m}^{(\gamma_k),(\delta_k)}f(x)
&=\int_0^1 \beta G_{m,m}^{m,0}\left[u^\beta\left|\begin{matrix}
	(\gamma_k+\delta_k+1-1/\beta)_{1,m}\\
	(\gamma_k+1-1/\beta)_{1,m}
\end{matrix}\right.\right]f(xu)\mathrm{d}u\\
&=\int_0^1  G_{m,m}^{m,0}\left[\sigma\left|\begin{matrix}
	(\gamma_k+\delta_k+1-1/\beta)_{1,m}\\
	(\gamma_k+1-1/\beta)_{1,m}
\end{matrix}\right.\right]f(x\sigma^{1/\beta})\sigma^{1/\beta-1}\mathrm{d}\sigma
=I_{\beta,m}^{(\gamma_k),(\delta_k)}f(x).
\end{align*}
Similarly, we have 
\[
K_{(\beta,\cdots,\beta),(\beta,\cdots,\beta),n}^{(\tau_k),(\alpha_k)}f(x)
=W_{\beta,n}^{(\gamma_k),(\delta_k)}f(x).
\]
In addition, if we just let $\lambda_k=\beta_k$ $(k=1,\cdots,m)$ in \eqref{GalueFIO-Def-1} and $\xi_k=\varepsilon_k$ $(k=1,\cdots,n)$ in \eqref{GalueFIO-Def-2}, respectively, we obtain the operators $I_{(\beta_k),m}^{(\gamma_k),(\delta_k)}$ and $K_{(\varepsilon_k),n}^{(\tau_k),(\alpha_k)}$ introduced by Kalla and Kiryakova (see \cite{Kalla-Kiryakova-1990a} and \cite{Kalla-Kiryakova-1990b}).  

Since the digital copy of the paper \cite{Galue-Kalla-Srivastava-1993} is not easily available, the interested readers may refer to \cite{Galue-Kiryakova-Kalla-1993}, \cite{Kiryakova-1997}, \cite{Kiryakova-2010} and \cite{Srivastava-Saxena-2001} for details.

\section{Preliminaries}

In this section, we present some results concerning the asymptotic behavior of the $H$-function, which will play a crucial role in the process of proving the main conclusions in Section \ref{MainResults}.

The asymptotic behaviour of the $H$-function near $z=0$ is well-known. Under the condition $\Delta=0$, we have \cite[p. 20, Corollary 1.11.1]{Kilbas-Saigo-Book-2004}
\begin{equation}\label{H-function-limit-1}
	H_{p, q}^{m, n}\left[z \left|\begin{matrix}
		(a_i, A_i)_{1,p} \\
		(b_j, B_j)_{1,q}
	\end{matrix}\right.\right]=\mathcal{O}\big(z^{\rho^*}\big)~(z \rightarrow 0),
\end{equation}
where 
\[
\rho^*:=\min _{1 \leq k \leq m}\left[\frac{\operatorname{Re}\left(b_k\right)}{B_k}\right].
\]

However, the asymptotic behavior of the $H$-function near $z=\delta$, where $\delta$ is given by \eqref{H-function-Condition-2}, is much more complex and has only recently been clearly characterized. When $\Delta=0$, we have (see \cite[p. 350, Theorem 1]{Karp-2020} and \cite[p. 22, Remark 1]{Kiryakova-PanevaKonovska-2024})
\begin{equation}\label{H-function-limit-2}
	H_{m, m}^{m, 0}\left[z\left|
	\begin{matrix}
		(a_k, A_k)_1^m\\
		(b_k, B_k)_1^m
	\end{matrix}\right.\right]=\mathcal{O}\left((1-z)^{\mu-1}\right)~(z \rightarrow 1^-)
\end{equation}
where
\begin{equation}\label{H-function-limit-2-Condition}
\mu=\sum_{k=1}^m(a_k-b_k)>0.
\end{equation}
In particular, if $A_k=B_k=1$ $(k=1,\cdots,m)$, we obtain the more familiar result (see \cite[p. 317, Eq. (A.22)]{Kiryakova-book-1994}) 
\begin{equation}\label{G-function-limit-1}
	G_{m, m}^{m, 0}\left[z\left|
	\begin{matrix}
		(a_k)_1^m\\
		(b_k)_1^m
	\end{matrix}\right.\right]=\mathcal{O}\left((1-z)^{\mu-1}\right)~(z \rightarrow 1^-)
\end{equation}
where $\mu$ is given by \eqref{H-function-limit-2-Condition}. The estimate \eqref{G-function-limit-1} can be derived from N\o rlund's famous formula (see \cite[p. 344, Eq. (2)]{Karp-Prilepkina-2017}). 

\section{Main results}\label{MainResults}

In this section, we establish the $H^1$-boundedness of the multiple Erd\'{e}lyi-Kober fractional integral operators $I_{(\beta_k),(\lambda_k), m}^{(\gamma_k),(\delta_k)}$ and $K_{(\varepsilon_k),(\xi_k), n}^{(\tau_k),(\alpha_k)}$. 
\begin{theorem}\label{MainResult-Th-1}
	Let $m \in \mathbb{Z}_{\geq1},\beta_k,\lambda_k>0,\delta_k \geq 0,\gamma_k \in \mathbb{R}$ $(k=1,\ldots,m)$. If $(\beta_k)_{1,m},(\lambda_k)_{1,m},(\delta_k)_{1,m} $ and $ (\gamma_k)_{1,m}$ satisfy the following conditions:
	\begin{itemize}
	\item[\emph{(1)}] $\displaystyle 
			\sum_{k=1}^m \frac{1}{\lambda_k}=\sum_{k=1}^m \frac{1}{\beta_k}$;
	\item[\emph{(2)}] $\displaystyle 
			\sum_{k=1}^m \delta_k>0$;
	\item[\emph{(3)}] $\displaystyle\min_{1\leq k\leq m}(\gamma_k+1) \lambda_k>1$.
	\end{itemize}
	then for any $f \in H^1$,we have
	\begin{equation}\label{MainResult-Th-1-1}
		\left\|I_{(\beta_k),(\lambda_k), m}^{(\gamma_k),(\delta_k)} f\right\|_{H^1} \leq\mathsf{k}_1\|f\|_{H^1},
	\end{equation}
	where
	\[
	\mathsf{k}_1\equiv 
	\int_0^1\left|H_{m, m}^{m, 0}\left[\sigma \left|\begin{matrix}
		(\gamma_k+\delta_k+1-1/\beta_k,1/\beta_k)_{1,m}\\
		(\gamma_k+1-1/\lambda_k,1/\lambda_k)_{1,m}
	\end{matrix}\right.\right]\right| \frac{\mathrm{d}\sigma}{\sigma}.
	\]
\end{theorem}
\begin{proof}
	Let $f \in H^1$, for any fixed $t>0$, we consider
	\[
		\int_{\mathbb{R}} \int_0^1\left|P_t(y-x) H_{m, m}^{m, 0}\left[\sigma \left|\begin{matrix}
			(\gamma_k+\delta_k+1-1/\beta_k,1/\beta_k)_{1,m}\\
			(\gamma_k+1-1/\lambda_k,1/\lambda_k)_{1,m}
		\end{matrix}\right.\right] f(x \sigma)\right| \mathrm{d} \sigma \mathrm{d} x.
	\]
	As 
	\begin{equation}\label{MainResult-Th-1-Proof-1}
	|P_t(y-x)|=\frac{1}{\pi} \cdot \frac{t}{t^2+(y-x)^2} \leq \frac{1}{\pi}\cdot t^{-1}
	\end{equation}
	for $x,y\in\mathbb{R}$ and $t>0$, we have   
	\begin{align*}
			&\int_{\mathbb{R}} \int_0^1\left|P_t(y-x) H_{m, m}^{m, 0}\left[\sigma \left|\begin{matrix}
				(\gamma_k+\delta_k+1-1/\beta_k,1/\beta_k)_{1,m}\\
				(\gamma_k+1-1/\lambda_k,1/\lambda_k)_{1,m}
			\end{matrix}\right.\right] f(x \sigma)\right| \mathrm{d}\sigma \mathrm{d}x \\
			&\leq  \frac{1}{\pi}\cdot t^{-1} \int_{\mathbb{R}} \int_0^1\left|H_{m, m}^{m, 0}\left[\sigma \left|\begin{matrix}
				(\gamma_k+\delta_k+1-1/\beta_k,1/\beta_k)_{1,m}\\
				(\gamma_k+1-1/\lambda_k,1/\lambda_k)_{1,m}
			\end{matrix}\right.\right] f(x \sigma)\right| \mathrm{d}\sigma \mathrm{d}x\\
			&=\frac{1}{\pi}\cdot t^{-1} \left(\int_{\mathbb{R}}|f(z)|\mathrm{d}z\right)
			\left(\int_0^1\left|H_{m, m}^{m, 0}\left[\sigma \left|\begin{matrix}
				(\gamma_k+\delta_k+1-1/\beta_k,1/\beta_k)_{1,m}\\
				(\gamma_k+1-1/\lambda_k,1/\lambda_k)_{1,m}
			\end{matrix}\right.\right]\right|\frac{\mathrm{d}\sigma}{\sigma}\right).
	\end{align*}
	Since $f \in H^1 \subset L^1$, we can write
	\begin{align*}
			&\int_{\mathbb{R}} \int_0^1\left|P_t(y-x) H_{m, m}^{m, 0}\left[\sigma \left|\begin{matrix}
				(\gamma_k+\delta_k+1-1/\beta_k,1/\beta_k)_{1,m}\\
				(\gamma_k+1-1/\lambda_k,1/\lambda_k)_{1,m}
			\end{matrix}\right.\right] f(x \sigma)\right| \mathrm{d}\sigma \mathrm{d}x \\
			&\leq \frac{1}{\pi}\cdot t^{-1}\|f\|_{L_1}(I_1+I_2),
	\end{align*}
	where
	\[
		I_1:=\int_0^{\frac{1}{2}}\left|H_{m, m}^{m, 0}\left[\sigma \left|\begin{matrix}
			(\gamma_k+\delta_k+1-1/\beta_k,1/\beta_k)_{1,m}\\
			(\gamma_k+1-1/\lambda_k,1/\lambda_k)_{1,m}
		\end{matrix}\right.\right]
		\right|\frac{\mathrm{d} \sigma}{\sigma}
	\]
	and
	\[
		I_2:=\int_{\frac{1}{2}}^1\left|H_{m, m}^{m, 0}\left[\sigma \left|\begin{matrix}
			(\gamma_k+\delta_k+1-1/\beta_k,1/\beta_k)_{1,m}\\
			(\gamma_k+1-1/\lambda_k,1/\lambda_k)_{1,m}
		\end{matrix}\right.\right]\right|\frac{\mathrm{d} \sigma}{\sigma}.
	\]
	
	Under the condition (1), the function $H_{m,m}^{m,0}$ is sometimes called \emph{delta neutral} \cite[p. 346]{Karp-Prilepkina-2017}. 
	For the integral $I_1$, according to \eqref{H-function-limit-1} and condition (3), we have
	\[
		H_{m, m}^{m, 0}\left[\sigma \left|\begin{matrix}
			(\gamma_k+\delta_k+1-1/\beta_k,1/\beta_k)_{1,m}\\
			(\gamma_k+1-1/\lambda_k,1/\lambda_k)_{1,m}
		\end{matrix}\right.\right]=\mathcal{O}\left(\sigma^{\rho_0}\right)~(\sigma \rightarrow 0^{+})
	\]
	where $\displaystyle\rho_0=\min_{1\leq k \leq m}(\gamma_k+1)\lambda_k-1>0$. Therefore there exists a positive number $C_1$ so that 
	\[
		I_1\leq C_1 \int_0^{\frac{1}{2}} \sigma^{\rho_0-1} \mathrm{d} \sigma<\infty.
	\]
	For the integral $I_2$, according to \eqref{H-function-limit-2} and conditions (1) and (2), we obtain
	\[
		H_{m, m}^{m, 0}\left[\sigma \left|\begin{matrix}
			(\gamma_k+\delta_k+1-1/\beta_k,1/\beta_k)_{1,m}\\
			(\gamma_k+1-1/\lambda_k,1/\lambda_k)_{1,m}
		\end{matrix}\right.\right]=\mathcal{O} \left((1-\sigma)^{\mu_0-1}\right)~(\sigma\rightarrow 1^-)
	\]
	where
	\[ 
	\mu_0=\sum\limits_{k=1}^m \delta_k+\sum\limits_{k=1}^m\left(\frac{1}{\lambda_k}-\frac{1}{\beta_k}\right)>0.
	\]
	Hence there exists a positive number $C_2$ so that 
	\[
		I_2 \leq C_2\int_{\frac{1}{2}}^1(1-\sigma)^{\mu_0-1} \sigma^{-1} \mathrm{d}\sigma<\infty.
	\]
	
	Thus we are allowed to use Fubini's theorem to obtain
\begin{align*}
	P_t *I_{(\beta_k),(\lambda_k), m}^{(\gamma_k),(\delta_k)} f
	&=\int_{\mathbb{R}} \int_0^1 P_t(y-x) H_{m, m}^{m, 0}\left[\sigma \left|\begin{matrix}
	(\gamma_k+\delta_k+1-1/\beta_k,1/\beta_k)_{1,m}\\
	(\gamma_k+1-1/\lambda_k,1/\lambda_k)_{1,m}
	\end{matrix}\right.\right] f(x \sigma) \mathrm{d}\sigma \mathrm{d}x\\
	&=\int_0^1 H_{m, m}^{m, 0}\left[\sigma \left|\begin{matrix}
	(\gamma_k+\delta_k+1-1/\beta_k,1/\beta_k)_{1,m}\\
	(\gamma_k+1-1/\lambda_k,1/\lambda_k)_{1,m}
	\end{matrix}\right.\right]\left(\int_{\mathbb{R}} P_t(y-x) f(x \sigma)\mathrm{d}x\right)\mathrm{d}\sigma.
\end{align*}
The inner integral can also be represented as
\[
\int_{\mathbb{R}} P_t(y-x) f(x \sigma)\mathrm{d}x
=\int_{\mathbb{R}} P_t(y-x) D_{1/\sigma}f(x)\mathrm{d}x
=P_t \ast (D_{1/\sigma}f).
\]
From \eqref{Def-MP} we have 
\begin{equation}\label{MainResult-Th-1-Proof-2}
|P_t \ast (D_{1/\sigma}f)|\leq M_P D_{1/\sigma}f,
\end{equation}
and thus
\[
\left|P_t *I_{(\beta_k),(\lambda_k), m}^{(\gamma_k),(\delta_k)} f\right|
\leq \int_0^1 \left|H_{m, m}^{m, 0}\left[\sigma \left|\begin{matrix}
	(\gamma_k+\delta_k+1-1/\beta_k,1/\beta_k)_{1,m}\\
	(\gamma_k+1-1/\lambda_k,1/\lambda_k)_{1,m}
\end{matrix}\right.\right]\right|\cdot
M_P D_{1/\sigma}f(y)
\mathrm{d}\sigma.
\]
	
For any $\lambda \in (0,\infty)$ and $f \in H^1$, we have
\begin{align*}
		(P_t *D_\lambda f)(x)
		&=\int_{\mathbb{R}}  P_t(x-y)f\left(\frac{y}{\lambda}\right)\mathrm{d}y \\
		&=\int_{\mathbb{R}} f\left(\frac{x}{\lambda}-z\right) P_t(\lambda z) \cdot \lambda \mathrm{d}z ~~~ (y=x-\lambda z)\\
		& =\int_{\mathbb{R}} f\left(\frac{x}{\lambda}-z\right) P_{t/\lambda}(z) \mathrm{d}z
		=(f \ast P_{ t/\lambda})\left(\frac{x}{\lambda}\right).
\end{align*}  
and thus
\begin{align}\label{MainResult-Th-1-Proof-3}
	M_P D_\lambda f(x) 
	&=\sup _{t \geq 0}|(P_t *D_\lambda f)(x)|\notag\\
	&=\sup _{t \geq 0}\left|\left(f \ast P_{t/\lambda}\right)\left(\frac{x}{\lambda}\right)\right|
	=M_P f\left(\frac{x}{\lambda}\right)
	=D_\lambda M_P f(x).
\end{align}
We now have
\begin{equation}\label{MainResult-Th-1-Proof-4}
	\left|P_t * I_{(\beta_k),(\lambda_k), m}^{(\gamma_k),(\delta_k)} f\right| \leq \int_0^1\left|H_{m, m}^{m, 0}\left[\sigma \left|\begin{matrix}
	(\gamma_k+\delta_k+1-1/\beta_k,1/\beta_k)_{1,m}\\
	(\gamma_k+1-1/\lambda_k,1/\lambda_k)_{1,m}
	\end{matrix}\right.\right]\right| \cdot D_{1/\sigma} M_P f(y) \mathrm{d}\sigma.
\end{equation}
By taking supremum over $t>0$ on both side of \eqref{MainResult-Th-1-Proof-4} gives
\begin{align*}
	M_P I_{(\beta_k),(\lambda_k), m}^{(\gamma_k),(\delta_k)} f&=\sup _{t>0}\left|P_t \ast I_{(\beta_k),(\lambda_k), m}^{(\gamma_k),(\delta_k)} f\right| \\
	&\leq
	\int_0^1\left|H_{m, m}^{m, 0}\left[\sigma \left|\begin{matrix}
	(\gamma_k+\delta_k+1-1/\beta_k,1/\beta_k)_{1,m}\\
	(\gamma_k+1-1/\lambda_k,1/\lambda_k)_{1,m}
	\end{matrix}\right.\right]\right| \cdot D_{1/\sigma} M_P f(y)\mathrm{d} \sigma.
\end{align*}
	
Finally, we obtain
\begin{align*}
	\left\|I_{(\beta_k),(\lambda_k), m}^{(\gamma_k),(\delta_k)} f\right\|_{H^1}
	&=\int_{\mathbb{R}}\left|M_P I_{(\beta_k),(\lambda_k), m}^{(\gamma_k),(\delta_k)} f(y)\right|\mathrm{d}y\\
	&\leq \int_{\mathbb{R}} \int_0^1\left|H_{m, m}^{m, 0}\left[\sigma \left|\begin{matrix}
	(\gamma_k+\delta_k+1-1/\beta_k,1/\beta_k)_{1,m}\\
	(\gamma_k+1-1/\lambda_k,1/\lambda_k)_{1,m}
	\end{matrix}\right.\right]\right| \cdot D_{1/\sigma} M_P f(y) \mathrm{d}\sigma \mathrm{d}y\\
	&=\int_{\mathbb{R}} \int_0^1\left|H_{m, m}^{m, 0}\left[\sigma \left|\begin{matrix}
		(\gamma_k+\delta_k+1-1/\beta_k,1/\beta_k)_{1,m}\\
		(\gamma_k+1-1/\lambda_k,1/\lambda_k)_{1,m}
	\end{matrix}\right.\right]\right| \cdot M_P f(\sigma y) \mathrm{d}\sigma \mathrm{d}y\\
	&=\left(\int_{\mathbb{R}}M_P f(y)\mathrm{d}y\right)\left(\int_0^1\left|H_{m, m}^{m, 0}\left[\sigma \left|\begin{matrix}
		(\gamma_k+\delta_k+1-1/\beta_k,1/\beta_k)_{1,m}\\
		(\gamma_k+1-1/\lambda_k,1/\lambda_k)_{1,m}
	\end{matrix}\right.\right]\right| \frac{\mathrm{d}\sigma}{\sigma} \right)\\
	&=\|f\|_{H^1}\left(\int_0^1\left|H_{m, m}^{m, 0}\left[\sigma \left|\begin{matrix}
		(\gamma_k+\delta_k+1-1/\beta_k,1/\beta_k)_{1,m}\\
		(\gamma_k+1-1/\lambda_k,1/\lambda_k)_{1,m}
	\end{matrix}\right.\right]\right| \frac{\mathrm{d}\sigma}{\sigma} \right),
\end{align*}
which is equivalent to \eqref{MainResult-Th-1-1}. This completes the proof.
\end{proof}
\begin{remark}
	By letting $\lambda_k=\beta_k=\beta$ $(k=1,\cdots,m)$, Theorem \ref{MainResult-Th-1} reduces to the first assertion of Theorem \ref{HoTh}.
\end{remark}

\begin{theorem}\label{MainResult-Th-2}
	Let $n \in \mathbb{Z}_{\geq0},\varepsilon_k,\xi_k>0,\alpha_k \geq 0,\tau_k \in \mathbb{R}$ $(k=1,2,\ldots,n)$. If $(\varepsilon_k)_{1,n},(\xi_k)_{1,n},(\alpha_k)_{1,n} $ and $ (\tau_k)_{1,n}$ satisfy the following conditions:
\begin{itemize}
	\item[\emph{(1)}] $\displaystyle \sum_{k=1}^n \frac{1}{\xi_k}=\sum_{k=1}^n \frac{1}{\varepsilon_k}$;
	\item[\emph{(2)}] $\displaystyle \sum\limits_{k=1}^n \alpha_k>0$;
	\item[\emph{(3)}] $\displaystyle\min _{1 \leq k \leq n} \tau_k \xi_k>-1$,
\end{itemize}
then for any $f \in H^1$,we have
\begin{equation}\label{MainResult-Th-2-1}
	\left\|K_{(\varepsilon_k),(\xi_k), n}^{(\tau_k),(\alpha_k)} f\right\|_{H^1} 
	\leq \mathsf{k}_2\|f\|_{H^1},
\end{equation}
where
\[
\mathsf{k}_2\equiv \int_1^{\infty} \left| 
H_{n, n}^{n, 0}\left[\frac{1}{\sigma}\left|\begin{matrix}
	(\tau_k+\alpha_k+1/\varepsilon_k,1/\varepsilon_k)_{1,n}\\
	(\tau_k+1/\xi_k,1/\xi_k)_{1,n}
\end{matrix}\right.\right] \right|\frac{\mathrm{d}\sigma}{\sigma}.
\]
\end{theorem}
\begin{proof}
The proof is similar to that of Theorem \ref{MainResult-Th-1}. Let $f \in H^1$. Then for any fixed $t>0$, we consider
	\[
		\int_{\mathbb{R}} \int_1^{\infty} \left\lvert\, P_t(y-x) H_{n, n}^{n, 0}\left[\frac{1}{\sigma}\left|\begin{matrix}
			(\tau_k+\alpha_k+1/\varepsilon_k,1/\varepsilon_k)_{1,n}\\
			(\tau_k+1/\xi_k,1/\xi_k)_{1,n}
		\end{matrix}\right.\right] f(x\sigma) \right\rvert\, \mathrm{d}\sigma \mathrm{d}x.
	\]
	From \eqref{MainResult-Th-1-Proof-1}, we have
	\begin{align*}
		&\int_{\mathbb{R}} \int_1^{\infty} \left\lvert\, P_t(y-x) H_{n, n}^{n, 0}\left[\frac{1}{\sigma}\left|\begin{matrix}
				(\tau_k+\alpha_k+1/\varepsilon_k,1/\varepsilon_k)_{1,n}\\
				(\tau_k+1/\xi_k,1/\xi_k)_{1,n}
			\end{matrix}\right.\right] f(x \sigma) \right\rvert\, \mathrm{d}\sigma \mathrm{d}x\\
		&\leq \frac{1}{\pi}\cdot t^{-1} \int_{\mathbb{R}} \int_1^{\infty} \left\lvert\, H_{n, n}^{n, 0}\left[\frac{1}{\sigma}\left|\begin{matrix}
				(\tau_k+\alpha_k+1/\varepsilon_k,1/\varepsilon_k)_{1,n}\\
				(\tau_k+1/\xi_k,1/\xi_k)_{1,n}
			\end{matrix}\right.\right] f(x \sigma) \right\rvert\, \mathrm{d}\sigma \mathrm{d}x\\
		&\leq \frac{1}{\pi}\cdot t^{-1}\left(\int_{\mathbb{R}}|f(z)| d z\right) \cdot\left(\int_1^{\infty} \left\lvert\, H_{n, n}^{n, 0}\left[\frac{1}{\sigma}\left|\begin{matrix}
				(\tau_k+\alpha_k+1/\varepsilon_k,1/\varepsilon_k)_{1,n}\\
				(\tau_k+1/\xi_k,1/\xi_k)_{1,n}
			\end{matrix}\right.\right] \cdot \frac{1}{\sigma} \right\rvert\, \mathrm{d}\sigma\right).
	\end{align*}
Since $f \in H^1 \subset L_1$,we have
\begin{align*}
	&\int_{\mathbb{R}} \int_1^{\infty} \left\lvert\, P_t(y-x) H_{n, n}^{n, 0}\left[\frac{1}{\sigma}\left|\begin{matrix}
	(\tau_k+\alpha_k+1/\varepsilon_k,1/\varepsilon_k)_{1,n}\\
	(\tau_k+1/\xi_k,1/\xi_k)_{1,n}
	\end{matrix}\right.\right] f(x \sigma) \right\rvert\, \mathrm{d}\sigma \mathrm{d}x\\
	&\leq\frac{1}{\pi}\cdot t^{-1}\|f\|_{L_1}(I_3+I_4),
\end{align*}
where
\[
	I_3:=\int_1^2 \left|H_{n, n}^{n, 0}\left[\frac{1}{\sigma}\left|\begin{matrix}
	(\tau_k+\alpha_k+1/\varepsilon_k,1/\varepsilon_k)_{1,n}\\
	(\tau_k+1/\xi_k,1/\xi_k)_{1,n}
	\end{matrix}\right.\right]\right|
	\frac{\mathrm{d}\sigma}{\sigma}
\]
and
\[
	I_4:=\int_2^{\infty} \left| H_{n, n}^{n, 0}\left[\frac{1}{\sigma}\left|\begin{matrix}
	(\tau_k+\alpha_k+1/\varepsilon_k,1/\varepsilon_k)_{1,n}\\
	(\tau_k+1/\xi_k,1/\xi_k)_{1,n}
	\end{matrix}\right.\right]\right|
	\frac{\mathrm{d}\sigma}{\sigma}.
\]
	
For the integral $I_3$, according to \eqref{H-function-limit-2} and conditions (1) and (2), we find
	\[
		H_{n, n}^{n, 0}\left[\frac{1}{\sigma}\left|\begin{matrix}
			(\tau_k+\alpha_k+1/\varepsilon_k,1/\varepsilon_k)_{1,n}\\
			(\tau_k+1/\xi_k,1/\xi_k)_{1,n}
		\end{matrix}\right.\right]= \mathcal{O} \left(\Big(1-\frac{1}{\sigma}\Big)^{\mu_1-1}\right)~(\sigma\rightarrow 1^+)
	\]
	where 
	\[
	\mu_1=\sum\limits_{k=1}^n \alpha_k+\sum\limits_{k=1}^n\left(\frac{1}{\varepsilon_k}-\frac{1}{\xi_k}\right)>0.
	\]
	Therefore there exists a positive number $C_3$ so that 
	\[
		I_3\leq C_3 \int_1^2(\sigma-1)^{\mu_1-1} \sigma^{-\mu_1}\mathrm{d}\sigma<\infty.
	\]
	For the integral $I_4$, according to \eqref{H-function-limit-1} and condition (3), we get
	\[
		H_{n, n}^{n, 0}\left[\frac{1}{\sigma}\left|\begin{matrix}
			(\tau_k+\alpha_k+1/\varepsilon_k,1/\varepsilon_k)_{1,n}\\
			(\tau_k+1/\xi_k,1/\xi_k)_{1,n}
		\end{matrix}\right.\right]= \mathcal{O} \left(\sigma^{-\rho_1}\right) ~ (\sigma\rightarrow \infty).
	\]
	where $\displaystyle 
	\rho_1=\min_{1\leq k\leq n} \tau_k \xi_k+1>0$. Thus there exists a positive number $C_4$ so that 
	\[
		I_4\leq C_4\int_2^{\infty} \sigma^{-\rho_1-1} \mathrm{d}\sigma<\infty.
	\]
	
	Thus, we are allowed to use Fubini's theorem to obtain
	\begin{align*}
		P_t \ast K_{(\varepsilon_k),(\xi_k), n}^{(\tau_k),(\alpha_k)} f 
		&=\int_{\mathbb{R}} \int_1^{\infty} P_t(y-x) H_{n, n}^{n, 0}\left[\frac{1}{\sigma}\left|\begin{matrix}
				(\tau_k+\alpha_k+1/\varepsilon_k,1/\varepsilon_k)_{1,n}\\
				(\tau_k+1/\xi_k,1/\xi_k)_{1,n}
			\end{matrix}\right.\right] f(x \sigma) \mathrm{d}\sigma\mathrm{d} x\\
		&=\int_1^{\infty} H_{n, n}^{n, 0}\left[\frac{1}{\sigma}\left|\begin{matrix}
				(\tau_k+\alpha_k+1/\varepsilon_k,1/\varepsilon_k)_{1,n}\\
				(\tau_k+1/\xi_k,1/\xi_k)_{1,n}
			\end{matrix}\right.\right]\int_{\mathbb{R}} P_t(y-x) f(x \sigma) \mathrm{d}x \mathrm{d}\sigma\\
		&=\int_1^{\infty} H_{n, n}^{n, 0}\left[\frac{1}{\sigma}\left|\begin{matrix}
				(\tau_k+\alpha_k+1/\varepsilon_k,1/\varepsilon_k)_{1,n}\\
				(\tau_k+1/\xi_k,1/\xi_k)_{1,n}
		\end{matrix}\right.\right]\cdot P_t *(D_{1/\sigma} f) \mathrm{d}\sigma.  
	\end{align*}
	
	From \eqref{MainResult-Th-1-Proof-2} and \eqref{MainResult-Th-1-Proof-3} we can see that
	\begin{align}\label{MainResult-Th-2-Proof-1}
		\left|P_t \ast K_{(\varepsilon_k),(\xi_k), n}^{(\tau_k),(\alpha_k)} f\right| 
		&\leq \int_1^{\infty} \left|H_{n, n}^{n, 0}\left[\frac{1}{\sigma}\left|\begin{matrix}
			(\tau_k+\alpha_k+1/\varepsilon_k,1/\varepsilon_k)_{1,n}\\
			(\tau_k+1/\xi_k,1/\xi_k)_{1,n}
		\end{matrix}\right.\right]\right|M_P D_{1/\sigma}f(y)\mathrm{d}\sigma\notag\\
		&=\int_1^{\infty} \left|H_{n, n}^{n, 0}\left[\frac{1}{\sigma}\left|\begin{matrix}
			(\tau_k+\alpha_k+1/\varepsilon_k,1/\varepsilon_k)_{1,n}\\
			(\tau_k+1/\xi_k,1/\xi_k)_{1,n}
		\end{matrix}\right.\right]\right|
		D_{1/\sigma} M_P f(y)\mathrm{d}\sigma.
	\end{align}
	By taking supremum over $t>0$ on both side of \eqref{MainResult-Th-2-Proof-1}, we have
	\begin{align*}
		M_P K_{(\varepsilon_k),(\xi_k), n}^{(\tau_k),(\alpha_k)} f &=\sup _{t>0}\left|P_t \ast K_{(\varepsilon_k),(\xi_k), n}^{(\tau_k),(\alpha_k)} f\right|\\
			&\leq\int_1^{\infty} \left| H_{n, n}^{n, 0}\left[\frac{1}{\sigma}\left|\begin{matrix}
			(\tau_k+\alpha_k+1/\varepsilon_k,1/\varepsilon_k)_{1,n}\\
			(\tau_k+1/\xi_k,1/\xi_k)_{1,n}
		\end{matrix}\right.\right] \right| D_{1/\sigma} M_P f(y)\mathrm{d}\sigma.
	\end{align*}
	
Finally, we obtain
	\begin{align*}
			\left\|K_{(\varepsilon_k),(\xi_k), n}^{(\tau_k),(\alpha_k)} f\right\|_{H^1}
			&=\int_{\mathbb{R}}\left|M_P K_{(\epsilon_k),(\xi_k), n}^{(\tau_k),(\alpha_k)} f(y)\right|\mathrm{d}y\\
			&\leq \int_{\mathbb{R}} \int_1^{\infty} \left| H_{n, n}^{n, 0}\left[\frac{1}{\sigma}\left|\begin{matrix}
				(\tau_k+\alpha_k+1/\varepsilon_k,1/\varepsilon_k)_{1,n}\\
				(\tau_k+1/\xi_k,1/\xi_k)_{1,n}
			\end{matrix}\right.\right] \right| D_{1/\sigma} M_P f(y)\mathrm{d}\sigma \mathrm{d}y\\
			&=\int_{\mathbb{R}} \int_1^{\infty} \left| H_{n, n}^{n, 0}\left[\frac{1}{\sigma}\left|\begin{matrix}
				(\tau_k+\alpha_k+1/\varepsilon_k,1/\varepsilon_k)_{1,n}\\
				(\tau_k+1/\xi_k,1/\xi_k)_{1,n}
			\end{matrix}\right.\right] \right|\cdot 
			M_P f(x \sigma) \mathrm{d}\sigma \mathrm{d}x\\
			&=\left(\int_{\mathbb{R}}M_P f(y)\mathrm{d}y\right)\int_1^{\infty} \left| H_{n, n}^{n, 0}\left[\frac{1}{\sigma}\left|\begin{matrix}
				(\tau_k+\alpha_k+1/\varepsilon_k,1/\varepsilon_k)_{1,n}\\
				(\tau_k+1/\xi_k,1/\xi_k)_{1,n}
			\end{matrix}\right.\right] \right| \frac{\mathrm{d}\sigma}{\sigma} \\
			&=\left(\int_1^{\infty} \left| H_{n, n}^{n,0}\left[\frac{1}{\sigma}\left|\begin{matrix}
				(\tau_k+\alpha_k+1/\varepsilon_k,1/\varepsilon_k)_{1,n}\\
				(\tau_k+1/\xi_k,1/\xi_k)_{1,n}
			\end{matrix}\right.\right] \right|\frac{\mathrm{d}\sigma}{\sigma}\right)\cdot \|f\|_{H_1},
	\end{align*}
which is equivalent to \eqref{MainResult-Th-2-1}. This completes the proof. 
\end{proof}

\begin{remark}
	By letting $\xi_k=\varepsilon_k=\beta$ $(k=1,\cdots,n)$, Theorem \ref{MainResult-Th-2} reduces to the second assertion of Theorem \ref{HoTh}.
\end{remark}

\section{Concluding remarks}

In this section, we provide some useful observations regarding the connections between the operators $I_{(\beta_k),(\lambda_k),m}^{(\gamma_k),(\delta_k)}$ and $K_{(\varepsilon_k),(\xi_k),n}^{(\tau_k),(\alpha_k)}$ and the Hausdorff operators of the form (see, for example, \cite[p. 1647]{Chen-Dai-Fan-Zhu-2018} and \cite[p, 143]{Chen-Fan-Zhu-2016}):
\begin{equation}\label{Hausdorff operator}
		H_{\phi}f(x):=\int_{0}^{\infty}\phi(u)f(xu)\mathrm{d}x.
\end{equation}
For the operator $H_{\phi}$, we have the following result.
\begin{theorem}[{\cite[p. 144]{Chen-Fan-Zhu-2016}}]\label{Theorem-A}
	If $\phi(u)\geq0$, then $H_{\phi}$ is bounded on $H^1(\mathbb{R})$ if and only if 
	\[
	\phi(u)u^{-1/p}\in L^1(0,\infty).
	\] 
\end{theorem}
Note that the assumption $\phi(u)\geq0$ in Theorem \ref{Theorem-A} \emph{cannot} be removed \cite[p. 144]{Chen-Fan-Zhu-2016}. Therefore, while operators $I_{(\beta_k),(\lambda_k),m}^{(\gamma_k),(\delta_k)}$ and $K_{(\varepsilon_k),(\xi_k),n}^{(\tau_k),(\alpha_k)}$ can both be interpreted as Hausdorff operators defined by \eqref{Hausdorff operator}, the theory of these Hausdorff operators \emph{cannot} be directly applied to  $I_{(\beta_k),(\lambda_k),m}^{(\gamma_k),(\delta_k)}$ and $K_{(\varepsilon_k),(\xi_k),n}^{(\tau_k),(\alpha_k)}$ unless certain positivity conditions are imposed. Furthermore, we need to point out that placing $I_{(\beta_k),(\lambda_k),m}^{(\gamma_k),(\delta_k)}$ and $K_{(\varepsilon_k),(\xi_k),n}^{(\tau_k),(\alpha_k)}$ within the framework of Hausdorff operators will not allow us to bypass the analysis of the $H$-function, which is precisely what makes these operators so interesting to study.

\section*{Acknowledgement}

The research of the second author is supported by National Natural Science Foundation of China (Grant No. 12001095).

\end{document}